\documentclass[10pt]{amsart}

\usepackage{amsthm,amsmath,stmaryrd,bbm,geometry,color}
\usepackage{amssymb}
\usepackage[english]{babel}
\usepackage[utf8]{inputenc}
\usepackage{graphicx}
\usepackage{verbatim}
\usepackage{enumitem}
\usepackage{mathtools}
\usepackage[titletoc]{appendix}
\usepackage{bm}
\usepackage[foot]{amsaddr}

\usepackage[dvipsnames]{xcolor}
\usepackage[hyperindex=true,frenchlinks=true,colorlinks=true,
citecolor=Mahogany,linkcolor=Red,urlcolor=Tan,linktocpage]{hyperref}

\usepackage{tikz}
\usetikzlibrary{shapes}
\usetikzlibrary{fit}
\usetikzlibrary{decorations.pathmorphing}

\title{Some remarks on the ergodic theorem for $U$-statistics}
\author{Herold Dehling, Davide Giraudo and Dalibor Voln\'y} 
\address{Fakultät für Mathematik, Ruhr-Universität Bochum, 44780 Bochum, Germany}
\email{herold.dehling@ruhr-uni-bochum.de}
\address{
Institut de Recherche Mathématique Avancée UMR 7501, Université de
Strasbourg and CNRS 7 rue René Descartes 67000 Strasbourg, France}
\email{dgiraudo@unistra.fr}
\address{University de Rouen, LMRS and CNRS UMR 6085.}
\email{dalibor.volny@univ-rouen.fr}
\date{\today}

\renewcommand{\leq}{\leqslant}
\renewcommand{\geq}{\geqslant}

\newtheorem{Theorem}{Theorem}[section]
\newtheorem{Proposition}[Theorem]{Proposition}
\newtheorem{Lemma}[Theorem]{Lemma}

\theoremstyle{remark}

\tikzstyle{Vertex}=[circle,draw=LimeGreen!80,fill=LimeGreen!8,
inner sep=1pt,minimum size=2mm,line width=1pt,font=\scriptsize]
\tikzstyle{Node}=[Vertex,draw=RoyalBlue!80,fill=RoyalBlue!8,inner sep=1.5pt]
\tikzstyle{Leaf}=[rectangle,draw=Black!70,fill=Black!16,
inner sep=0pt,minimum size=1mm,line width=1.25pt]
\tikzstyle{Edge}=[Maroon!80,cap=round,line width=1pt]
\tikzstyle{Mark1}=[draw=BrickRed!80,fill=BrickRed!8]
\tikzstyle{Mark2}=[draw=BurntOrange!80,fill=BurntOrange!8]
\tikzstyle{EdgeRew}=[->,RedOrange!80,cap=round,thick]

\newcommand{\Bca}{\mathcal{B}}

\newcommand{\Fca}{\mathcal{F}}

\newcommand \ens[1]{\left\{ #1\right\}}

\newcommand \R{\mathbb R}

\newcommand \N{\mathbb N}
\newcommand \PP{\mathbb P}

\newcommand{\E}[1]{\mathbb E\left[#1\right]}

\newcommand \abs[1]{\left|#1\right|}
\newcommand \eps{\varepsilon}

\newcommand{\pr}[1]{\left(#1\right)}

\newcommand{\ind}[1]{\mathbf{1}_{#1}}


\begin{document}

\begin{abstract}
In this note, we investigate the convergence of a $U$-statistic of order two having stationary ergodic data. We will find sufficient conditions for the almost sure and
$L^1$ convergence and present some counter-examples showing that the $U$-statistic itself might fail to converge: centering is needed as well as
finteness of $\sup_{j\geq 2}\E{\abs{h\pr{X_1,X_j}}}$.
\end{abstract}
\maketitle

\section{Introduction}
In this note, we investigate the validity of the $U$-statistics ergodic theorem, i.e. the almost sure convergence
\begin{equation}
  \frac{1}{\binom{n}{2}} \sum_{1\leq i<j\leq n} h(X_i,X_j) \longrightarrow \iint h(x,y) dF(x) dF(y),
\label{eq:u-ergodic}
\end{equation}
where $\pr{X_i}_{i\geq 1}$ is a stationary ergodic process with marginal distribution $F$, and $h\pr{x,y}$ is a symmetric kernel that is $F\times F$ integrable. Birkhoff's ergodic theorem establishes the analogous result for the time averages $\frac{1}{n} \sum_{i=1}^n f(X_i) $, while  Hoeffding \cite{Hoef:1961} established \eqref{eq:u-ergodic} for i.i.d. processes $(X_i)_{i\geq 1}$.  These two classical results naturally lead to the conjecture that \eqref{eq:u-ergodic} should hold without further assumptions, i.e. for all stationary ergodic processes $\pr{X_i}_{i\geq 1}$ and all $L_1(F\times F)$ functions $h(x,y)$.
Aaronson et al. \cite{MR1363941} proved a partial result in this direction, namely showing that \eqref{eq:u-ergodic} holds for all  $F\times F$ almost everywhere continuous and bounded kernels $h(x,y)$. At the same time, they presented counterexamples showing that \eqref{eq:u-ergodic} does not hold in full generality. One of their counterexamples is a bounded kernel where the set of discontinuities has positive $F\times F$ measure, while the other counterexample is an $F\times F$ almost everywhere continuous, but unbounded kernel.

The $U$-statistic ergodic theorem has subsequently been addressed by various authors, e.g. Arcones \cite{MR1624866}, Borovkova, Burton and Dehling
\cite{MR1687339}; see also the review paper by Borovkova, Burton and Dehling \cite{MR1979966}.
These papers provide both sufficient conditions for \eqref{eq:u-ergodic} to hold, as well as further counterexamples, both for stationary ergodic processes as well as under stronger mixing assumptions.
Most of the positive results also address other forms of convergence in \eqref{eq:u-ergodic} such as convergence in probability and $L^1$-convergence.  Arcones \cite{MR1624866} proved the ergodic theorem for absolutely regular processes under some moment assumptions. Borovkova, Burton and Dehling \cite{MR1979966} investigated convergence in probability in \eqref{eq:u-ergodic}, with a special focus on the kernel $h(x,y)=\log(|x-y|)$, which arises in connection with the Takens estimator for the correlation dimension.

A common feature of all these examples is that
they satisfy a modified version of the $U$-statistics ergodic theorem, namely
\begin{equation}
  \frac{1}{\binom{n}{2}} \sum_{1\leq i <j \leq n} \pr{ h(X_i,X_j) - \E{h\pr{X_i,X_j}} }\longrightarrow 0,
\label{eq:u-ergodic-2}
\end{equation}
assuming that $\E{\abs{h\pr{X_i,X_j}}}<\infty$ for all $i,j$.

It might thus seem natural to conjecture that \eqref{eq:u-ergodic-2} holds without further assumptions. In this note, we present a counterexample that disproves this conjecture. In addition, we will give a short proof of the $U$-statistics ergodic theorem for bounded $F\times F$-almost everywhere continuous kernels, and give a new condition for $L^1$-convergence.

\section{A short proof of the ergodic theorem for $U$-statistics}
In this note, we present a short proof of the $U$-statistics ergodic theorem that was first established in Aaronson et al \cite{MR1363941}. For the special case, when the process has values in $\R^k$, this proof is contained in Borovkova, Burton and Dehling \cite{MR1979966}. Here, we give the proof for processes with values in an arbitrary separable metric space.
\begin{Theorem}
Let $\pr{X_k}_{k\geq 0}$ be a stationary ergodic process with values in the separable metric space $S$ and marginal distribution $F$, and let $h:S\times S\rightarrow \R$ be a symmetric kernel that is bounded and $F\times F$-almost everywhere continuous. Then, as $n\rightarrow \infty$
\[
  \frac{1}{\binom{n}{2}} \sum_{1\leq i<j \leq n} h(X_i,X_j) \longrightarrow \iint h(x,y) dF(x) dF(y),
\]
almost surely.
\end{Theorem}
\begin{proof}
We define the empirical distribution of the first $n$ random variables
\[
  F_n=\frac{1}{n} \sum_{i=1}^n \delta_{X_i},
\]
where $\delta_x$ denotes the Dirac delta measure in $x$. For any $L_1(F)$-function $f\colon S\to \R$, we obtain by Birkhoff's ergodic theorem
\[
  \int_S f(x) \, dF_n(x) =\frac{1}{n}\sum_{i=1}^n f(X_i) \rightarrow \int_S f(x) dF(x),
\]
almost surely. This convergence holds in particular for any bounded measurable function $f\in C_b(S)$. Since $S$ is separable, there exists a countably family of functions $f_i\in C_b(S)$, $i\geq 1$,  that is convergence determining, i.e. that convergence of the integrals $\int f_i(x) d\mu_n(x) \rightarrow \int f_i(x) d\mu(x)$, for all $i\geq 1$, implies weak convergence of the probability measures $\mu_n$ to $\mu$. Now, up to a set of measure $0$, we get
\[
  \int_S f_i(x) \, dF_n(x) =\frac{1}{n}\sum_{j=1}^n f_i\pr{X_j} \rightarrow \int_S f_i(x) dF(x),
\]
for all $i\geq 1$, and thus $F_n \Rightarrow F$ weakly. This is in fact Varadarajan's argument \cite{MR94838} for the fact that the empirical distribution of i.i.d. data $X_1,\ldots, X_n$ converges weakly almost surely to the true distribution $F$.

By Theorem 3.2 (page 21) of Billingsley \cite{MR0233396}, we obtain  convergence of the empirical product measure
\[
  F_n \times F_n \Rightarrow F \times F,
\]
  except on a set of measure $0$. Thus, for any bounded $F\times F$-a.e. continuous function $h\colon S\times S\to\R$, we obtain by the portmanteau theorem
 \[
   \frac{1}{n^2} \sum_{1\leq i, j \leq n} h(X_i,X_j) =\iint h(x,y) dF_n(x) \, dF_n(y) \rightarrow \iint h(x,y) dF(x)\, dF(y),
 \]
almost surely. Since $h$ is bounded, we obtain $\frac{1}{n^2} \sum_{i=1}^n h(X_i,X_i) \rightarrow 0$, and thus
\[
   \frac{1}{n^2} \sum_{1\leq i\neq  j \leq n} h(X_i,X_j) \rightarrow \iint h(x,y) dF(x)\, dF(y),
\]
almost surely.
\end{proof}
\section{Convergence in $L^1$ in the ergodic theorem for $U$-statistics}
In this section, we present two sufficient conditions for the convergence in $L^1$
of a $U$-statistic to $\iint h\pr{x,y}dF\pr{x}dF\pr{y}$, where $F$ denotes the distribution of $X_0$. The first sufficient condition imposes a restriction on the continuity points
of the kernel combined with a uniform integrability assumption. The second sufficient condition imposes a restriction on
the joint distribution of vectors $\pr{X_0,X_k}, k\geq 1$, but no other assumption is required for the kernel $h$.

\begin{Theorem}\label{thm:conv_thm_ergo_densite_bornee}
 Let $\pr{X_i}_{i\geq 1}$ be a stationary ergodic sequence
 taking values in $\R^d$
 and let $h\colon \R^d\times\R^d\to \R$ be a measurable function
 such that the family $\ens{h\pr{X_1,X_j},j\geq 1}$ is uniformly integrable. Let $F$ be the distribution of $X_1$.
Assume that one of the following assumptions is satisfied:
\begin{enumerate}[label=(A.\arabic*)]
\item \label{itm:h_ae_cont} the function $h$ is $F \times F$ almost everywhere
continuous and symmetric.
\item\label{itm:bounded_dens} $
\int_{\R^d}\int_{\R^d}\abs{h\pr{x,y}}dF\pr{x}dF\pr{y} $ is finite,
the random variable  $X_0$ has a bounded density
with respect to the Lebesgue measure on $\R^d$ and
 for each $k\geq 1$,  the vector $\pr{X_0,X_k}$ has a density $f_k$ with respect to the Lebesgue measure of $\R^d\times\R^d$ and $\sup_{k\geq 1}\sup_{s,t\in\R^d}f_k\pr{s,t}$ is finite.
\end{enumerate}

 Then
\begin{equation}\label{eq:conv_thm_ergo_densite_bornee}
 \lim_{n\to\infty}\E{\abs{\frac{1}{\binom{n}{2}}\sum_{1\leq i<j\leq
n}h\pr{X_i,X_j}-\int_{\R^d}\int_{\R^d}h\pr{x,y}dF\pr{x}dF\pr{y}
   }}= 0.
\end{equation}

\end{Theorem}
\begin{proof}
Let us prove Theorem~\ref{thm:conv_thm_ergo_densite_bornee} under assumption \ref{itm:h_ae_cont}.
By Theorem~1 in \cite{MR1687339}, we know that
$\frac{1}{\binom{n}{2}}\sum_{1\leq i<j\leq
n}h\pr{X_i,X_j}\to
\int_{\R^d}\int_{\R^d}h\pr{x,y}dF\pr{x}dF\pr{y}$ in
probability. Then it suffices to notice that uniform integrability of
 $\ens{h\pr{X_1,X_j},j\geq 1}$  implies that of
$\ens{\frac{1}{\binom{n}{2}}\sum_{1\leq i<j\leq
n}h\pr{X_i,X_j},n\geq 2}$.

We will prove Theorem~\ref{thm:conv_thm_ergo_densite_bornee} under assumption \ref{itm:bounded_dens} in three steps: first we will show that \eqref{eq:conv_thm_ergo_densite_bornee} holds when $h$ is a product of indicator functions of Borel subsets of $\R^d$. Then we will show
the result by approximating the map $\pr{x,y}\in\R^d\times\R^d\mapsto
h\pr{x,y}\ind{[-R,R]^d}\pr{x}\ind{[-R,R]^d}\pr{y}\ind{\abs{h\pr{x,y}}\leq R}$
in $L^1\pr{\mathbb P_{\pr{X_0,X_k}}}$ uniformly with respect to $k$ by a linear
combination of products of indicator functions. Then we will
conclude by uniform integrability.

First step: assume that $h\pr{x,y}=\ind{A}\pr{x}\ind{B}\pr{y}$, where $A$ and $B$ are Borel subsets of $\R^d$. Observe that
\begin{align}
\frac{1}{\binom{n}{2}}\sum_{1\leq i<j\leq
n}h\pr{X_i,X_j}&=\frac{1}{\binom{n}{2}}\sum_{1\leq i<j\leq
n} \ind{A}\pr{X_i}\ind{B}\pr{X_j}\\
&=\frac{1}{\binom{n}{2}}\sum_{j=2}^n \ind{B}\pr{X_j}
\sum_{i=1}^{j-1}\ind{A}\pr{X_i}\\
&=\frac{1}{\binom{n}{2}}\sum_{j=2}^n \pr{j-1}\ind{B}\pr{X_j}
Y_j,
\end{align}
where
\begin{equation}
Y_j=\frac 1{j-1}\sum_{i=1}^{j-1}\ind{A}\pr{X_i}.
\end{equation}
Therefore, the following decomposition takes place:
\begin{equation}\label{eq:intermediate_step_ergo_den_bornee}
\frac{1}{\binom{n}{2}}\sum_{1\leq i<j\leq
n}h\pr{X_i,X_j} 
=\frac{1}{\binom{n}{2}}\sum_{j=2}^n \pr{j-1}\ind{B}\pr{X_j}
\pr{Y_j-\PP\pr{X_0\in A}}+\PP\pr{X_0\in A}\frac{1}{\binom{n}{2}}\sum_{j=2}^n \pr{j-1}\ind{B}\pr{X_j}
\end{equation}
Observe that by the ergodic theorem and the Lebesgue dominated convergence theorem, the first term of the right hand side of \eqref{eq:intermediate_step_ergo_den_bornee} converges to $0$ in $L^1$. Moreover, by  the ergodic theorem and a summation by parts,
\begin{equation}
\E{\abs{\frac{1}{\binom{n}{2}}\sum_{j=2}^n \pr{j-1} \ind{B}\pr{X_j}-\PP\pr{X_0\in B }}}\to 0,
\end{equation}
 hence we derive that
 \begin{equation}
 \lim_{n\to\infty}\E{\abs{\frac{1}{\binom{n}{2}}\sum_{1\leq i<j\leq
n}\ind{A}\pr{X_i}\ind{B}\pr{X_j}- \PP\pr{X_0\in A}\PP\pr{X_0\in B}
   }}=0
 \end{equation}
and $\PP\pr{X_0\in A}\PP\pr{X_0\in B}=\int_{\R^d}\int_{\R^d}h\pr{x,y}dF\pr{x}dF\pr{y}$.

Second step. Let $R>0$ be fixed and define
\begin{equation}\label{eq:def_of_hR}
h^{\pr{R}}\pr{x,y}=h\pr{x,y}\ind{[-R,R]^d}\pr{x}\ind{[-R,R]^d}\pr{y}
\ind{\abs{h\pr{x,y}}\leq R},
\end{equation}
which is integrable. By a standard result in measure theory, we know that
for each positive $\eps$, there exists an integer $N$, constants $c_1,\dots,c_N$ and
sets $A_{\eps,\ell},B_{\eps,\ell},1\leq \ell\leq N$, such that
\begin{equation}
\int_{\R^d\times\R^d}\abs{h^{\pr{R}}\pr{x,y} -h_{\eps}\pr{x,y} }d\lambda_d\pr{x}
d\lambda_d\pr{y}\leq\eps,
\end{equation}
where
\begin{equation}
h_{\eps}\pr{x,y}=\sum_{\ell=1}^N c_\ell
\ind{A_{\eps,\ell}}\pr{x}\ind{B_{\eps,\ell}}\pr{y}.
\end{equation}
Therefore, using stationarity and the fact that $\pr{X_i,X_j}$ has a density $f_{j-i}$ which is bounded by a constant $M$ independent of $\pr{i,j}$,
\begin{align*}
\E{\abs{h^{\pr{R}}\pr{X_i,X_j}-h_{\eps}\pr{X_i,X_j}}}&=
 \E{\abs{h^{\pr{R}}\pr{X_0,X_{j-i}}-h_{\eps}\pr{X_0,X_{j-i}}}}\\
 &=\int_{\R^d\times\R^d}\abs{h^{\pr{R}}\pr{x,y}-h_{\eps}\pr{x,y}}
 f_{j-i}\pr{x,y}d\lambda_d\pr{x}
d\lambda_d\pr{y}\leq M\eps
\end{align*}
and
\begin{multline}
\E{\abs{\int_{\R^d}\int_{\R^d}h^{\pr{R}}\pr{x,y}dF\pr{x}dF\pr{y}-\int_{\R^d}\int_{\R^d}h_{\eps}\pr{x,y}dF\pr{x}dF\pr{y}   }}\\
\leq \int_{\R^d}\int_{\R^d}\abs{h^{\pr{R}}\pr{x,y}-h_{\eps}\pr{x,y}}
f_{X_0}\pr{x}f_{X_0}\pr{y}\leq \sup_{t\in\R^d}f_{X_0}\pr{t}\eps.
\end{multline}
Consequently,
\begin{multline}
\E{\abs{\frac{1}{\binom{n}{2}}\sum_{1\leq i<j\leq
n}h^{\pr{R}}\pr{X_i,X_j}-  \int_{\R^d}\int_{\R^d}h^{\pr{R}}\pr{x,y}dF\pr{x}dF\pr{y}    }}
\\   \leq \E{\abs{\frac{1}{\binom{n}{2}}\sum_{1\leq i<j\leq
n}h_{\eps}\pr{X_i,X_j}-  \int_{\R^d}\int_{\R^d}h_{\eps}\pr{x,y}dF\pr{x}dF\pr{y}}}+\pr{M+\sup_{t\in\R^d}f_{X_0}\pr{t}}\eps.
\end{multline}
By the first step and the triangle inequality, we deduce that
for each positive $\eps$,
\begin{multline}\label{eq:conv_avec_hR}
\limsup_{n\to\infty}\E{\abs{\frac{1}{\binom{n}{2}}\sum_{1\leq i<j\leq
n}h^{\pr{R}}\pr{X_i,X_j}-  \int_{\R^d\times\R^d}h^{\pr{R}}\pr{x,y}dF\pr{x}dF\pr{y}}} \\ \leq  \pr{M+\sup_{t\in\R^d}f_{X_0}\pr{t}}\eps.
\end{multline}
hence \eqref{eq:conv_thm_ergo_densite_bornee} holds with $h$ replaced by $h_R$.
Third step: by uniform integrability, for each positive $\varepsilon$, there
exists $\delta$ such that for each $A$ satisfying $\PP\pr{A}<\delta$,
$\sup_{1\leq i<j}\E{\abs{h\pr{X_i,X_j}}\ind{A}}<\eps$. Let $R$ be such that
$\PP\pr{ X_1 \notin [-R,R]^d}<\delta$, $\sup_{j\geq 2}\E{\abs{h\pr{X_1,X_j}}
\ind{\ens{\abs{h\pr{X_1,X_j}}>R}}}<\eps$
and $\int_{\R^d}\int_{\R^d}\abs{h\pr{x,y}-h^{\pr{R}}\pr{x,y}}dF(x)dF(y)<\eps$.
Then for
$h^{\pr{R}}$ defined as in \eqref{eq:def_of_hR},
\begin{multline}
\E{\abs{h\pr{X_i,X_j}-h^{\pr{R}}\pr{X_i,X_j}}}\\
\leq \E{\abs{h\pr{X_i,X_j}}\pr{\ind{\ens{X_i \notin [-R,R]^d}}+
\ind{\ens{X_j \notin [-R,R]^d}}+\ind{\ens{\abs{h\pr{X_1,X_j}}>R}}   } }\leq
3\eps
\end{multline}
 and it follows that
\begin{multline}
\E{\abs{\frac{1}{\binom{n}{2}}\sum_{1\leq i<j\leq
n}h \pr{X_i,X_j}-  \int_{\R^d}\int_{\R^d}h\pr{x,y}dF\pr{x}dF\pr{y}    }}
\\   \leq \E{\abs{\frac{1}{\binom{n}{2}}\sum_{1\leq i<j\leq
n}h^{\pr{R}} \pr{X_i,X_j}-
\int_{\R^d}\int_{\R^d}h^{\pr{R}}\pr{x,y}dF\pr{x}dF\pr{y}    }}+4\eps,
\end{multline}
and we conclude by the second step. This ends the proof of Theorem~\ref{thm:conv_thm_ergo_densite_bornee}.
\end{proof}

\section{Examples of failure of the convergence of $U$-statistics }
Example 4.1 given in \cite{MR1363941} shows that there exists a stationary
ergodic sequence $\pr{X_i}_{i\geq 1}$ and a bounded measurable function
for which $\pr{\binom n2^{-1}\sum_{1\leq i<j\leq n}h\pr{X_i,X_j}}_{n\geq 2}$
converges, but not to the integral of $h\pr{x,y}$ with respect to the product of the law of $X_1$.

In a similar setting, we are able to formulate two examples, the first showing that
the sequence
$\pr{\binom n2^{-1}\sum_{1\leq i<j\leq n}h\pr{X_i,X_j}}_{n\geq 2}$
may fail to converge in probability even if $\abs{h\pr{X_i,X_j}}$ is bounded
by $1$, and the second one showing that a centered $U$-statistic $\pr{\binom
n2^{-1}\sum_{1\leq i<j\leq n}\pr{h\pr{X_i,X_j}-\E{h\pr{X_i,X_j}} }}_{n\geq 2}$
may also fail to converge in probability. \\
We consider the transformation $Tx = 2x \operatorname{mod} 1$ of the unit
interval $[0, 1)$ equipped with the Borel
sigma field $\Bca$ and Lebesgue measure $\lambda$. We define $X_0(x) =x$, $X_k(x) = T^kx$ and $U\colon L^1\to L^1$ by $ UY=Y\circ T$, $Y\in L^1$.

\subsection{Example 1: non-convergence of the $U$-statistics}

\begin{Proposition}\label{prop:example1}
There exists a strictly stationary ergodic sequence $\pr{X_i}_{i\geq 1}$ and
a bounded measurable symmetric function $h\colon\R^2\to\R$
such that the sequence
$\pr{\binom n2^{-1}\sum_{1\leq i<j\leq n}h\pr{X_i,X_j}}_{n\geq 2}$ does not converge in probability.
\end{Proposition}
\begin{proof}
Let $\pr{N_\ell}_{\ell \geq 1}$ and $\pr{N'_\ell}_{\ell \geq 0}$ be sequences of
positive  integers such that $N'_0=1$ and for $\ell\geq 1$,
$N_\ell<N'_{\ell}<N_{\ell+1}$   and
\begin{equation}\label{eq:assumption_Nl}
   N_\ell'/N_\ell \geq \ell, \quad  N_{\ell+1}/N'_\ell \to\infty.
\end{equation}
We define
\begin{equation}
  I = \bigcup_{\ell\geq 0}I_\ell,\quad I_\ell:=\ens{k\in\N : N'_\ell<k\leq N_{\ell+1}}
\end{equation}
\begin{equation}
   G = \bigcup_{k\in I}\ens{\pr{x, T^kx}: x\in [0,1) }
\end{equation}
and for $x,y \in [0,1)$,
$$
  h\pr{x,y} = \ind{G}\pr{x,y} +\ind{G}\pr{y,x} .
$$
 Since for $i<j$ and $k\geq 1$, the equality $T^{i}x =
T^{k+j}x$ can hold only for a countable set of $x$ (namely, the dyadic
rationals), we obtain for $1\leq i<j$ the identity
$h\pr{X_i,X_j}=\ind{G}\pr{X_i,X_j}$ almost surely. Moreover, by definition,
$\pr{X_i,X_j}\in G$ if and only if $T^{i+k}x = T^{j}x$ for some $k\in I$.
Almost surely, the latter identity holds if and only if $k=j-i$, and thus
$$
  h\pr{X_i,X_j} =\begin{cases} 1 &\text{if}\,\,\,  j-i \in I, \\
                              0 &\text{if}\,\,\,  j-i \in \Bbb N\setminus I.
  \end{cases}
$$
In particular, $ \abs{h\pr{X_i,X_j}}\leq 1$. By \eqref{eq:assumption_Nl} we
have
\begin{equation}\label{eq:convergence_of_some_subsequences}
  \frac1{N_\ell(N_\ell-1)} \sum_{1\leq i<j\leq N_\ell} h\pr{X_i,X_j} \to \frac12,
\quad
  \frac1{N'_\ell(N'_\ell-1)} \sum_{1\leq i<j\leq N'_\ell} h\pr{X_i,X_j} \to 0.
\end{equation}
Indeed, first observe that for each integer $n$,
\begin{align}
\sum_{1\leq i<j\leq n}h\pr{X_i,X_j}&=\sum_{j=2}^n\sum_{i=1}^{j-1}\ind{j-i\in I}\\
&=\sum_{j=2}^n\sum_{k=1}^{j-1}\ind{k\in I}\\
&=\sum_{k=1}^{n-1}\sum_{j=k+1}^n\ind{k\in I} =\sum_{k=1}^{n-1}\pr{n-k}\ind{k\in
I}\label{eq:expression_Ustats}
\end{align}
hence by definition of $I$, we get that for $\ell\geq 3$,
\begin{multline}\label{eq:step_convergence_subsequence_N_ell}
\frac 1{N_\ell\pr{N_\ell-1}}\sum_{1\leq i<j\leq N_\ell}h\pr{X_i,X_j}=
\frac 1{N_\ell\pr{N_\ell-1}}\sum_{u=1}^{\ell-2}\sum_{k\in I_u}
\pr{N_\ell-k}+\frac 1{N_\ell\pr{N_\ell-1}}\sum_{k\in I_{\ell-1}}\pr{N_\ell-k}\\
=:A_\ell+B_\ell.
\end{multline}
Note that bounding for $1\leq u\leq\ell-2$ the term $\sum_{k\in I_u}
\pr{N_\ell-k}$ by $N_\ell \operatorname{Card}\pr{I_u}$,
and $\operatorname{Card}\pr{I_u}$ by $ N_{u+1}-N_u$, we get
\begin{equation}
A_\ell\leq \frac{1}{ N_\ell-1 }\sum_{u=1}^{\ell-2}\pr{N_{u+1}-N_u}
\leq  \frac{N_{\ell-1}-N_1}{N_\ell-1}
\end{equation}
and using \eqref{eq:assumption_Nl}, we get
$A_\ell\to 0$.
 Moreover,
\begin{equation}
B_\ell=\frac 1{N_\ell\pr{N_\ell-1}}\sum_{k=N'_{\ell-1}+1}^{N_\ell}
\pr{N_\ell-k}=\frac 1{N_\ell\pr{N_\ell-1}}\sum_{j=0}^{N_\ell-N'_{\ell-1}-1}j\sim
\frac 12\frac{\pr{N_\ell-N'_{\ell-1}-1}^2}{N_\ell^2}
\end{equation}
hence $B_\ell\to 1/2$, which proves the first part of \eqref{eq:convergence_of_some_subsequences}. The second one follows from the observation that $\ens{1,\dots,N'_\ell-1}
\cap I_u$ is empty if $u\geq \ell-1$, which gives in view of \eqref{eq:expression_Ustats},
 \begin{align*}
   \frac1{N'_\ell(N'_\ell-1)} \sum_{1\leq i<j\leq N'_\ell} h\pr{X_i,X_j}&= \frac1{N'_\ell(N'_\ell-1)}\sum_{k=1}^{N'_\ell-1}
   \pr{N'_\ell-k}\ind{k\in I}\\
   &=\frac1{N'_\ell(N'_\ell-1)}\sum_{u=1}^{\ell-2}
   \sum_{k\in I_u}\pr{N'_\ell-k}\\
   &\leq\frac{1}{N'_\ell}\sum_{u=1}^{\ell-2}\pr{N_{u+1}-N'_u}\\
   &\leq\frac{1}{N'_\ell}\sum_{u=1}^{\ell-2}\pr{N_{u+1}-N_u}
   \leq \frac{N'_{\ell-1}}{N'_\ell},
 \end{align*}
 where the second inequality follows from $N_u<N'_u$,
 and $N'_{\ell-1}/N_\ell$ goes to $0$ by \eqref{eq:assumption_Nl}.
 \end{proof}
\subsection{Example 2: non-convergence of a centered  $U$-statistic}

\begin{Proposition}
There exists a strictly stationary ergodic sequence $\pr{X_i}_{i\geq 1}$ and
a  symmetric measurable function $h\colon\R^2\to\R$ such that
for each $i<j$, $\E{\abs{h\pr{X_i,X_j}}}$ is finite  but the sequence
$\pr{\binom n2^{-1}\sum_{1\leq i<j\leq n}\pr{h\pr{X_i,X_j}- \E{h\pr{X_i,X_j}}} }_{n\geq 2}$ does not converge in probability.
\end{Proposition}
Note that in this example, $\sup_{j\geq 2}\E{\abs{h\pr{X_1,X_j}}}$ is infinite. Moreover, the sequence
$$\pr{\binom n2^{-1}\sum_{1\leq i<j\leq n}\pr{h\pr{X_i,X_j}- \E{h\pr{X_i,X_j}}} }_{n\geq 2}$$ converges in distribution to a centered non-degenerated Gaussian random variable.
\begin{proof}
We take the same probability space and transformation as above. For
$k=1,2,\dots$ define
$$
  \bar G_k = T^{-k}\pr{[1/2, 1)}, \,\,\, G_k = \ens{\pr{x,T^kx} : x\in \bar G_k},
\,\,\,  h\pr{x,y} = \sum_{k=1}^\infty
a_k\pr{\ind{G_k}\pr{x,y}+\ind{G_k}\pr{y,x} }  ,
$$
where
\begin{equation}\label{eq:def_de_ak}
a_k=k^{3/2}-\pr{k-1}^{3/2} \mbox{ for }k\geq 2 \mbox{ and }a_1=1.
\end{equation}
 By similar arguments as in the proof of
Proposition~\ref{prop:example1}, the following equality holds
almost surely for each $1\leq i<j:$
$$
  h\pr{X_i, X_j}=  a_{j-i} U^jf, \mbox{ where } f= 1_{[1/2,
1)},
$$
hence $\E{h\pr{X_i, X_j}} = a_{j-i}/2$ and
\begin{equation}\label{eq:expression_Ustat_ex2}
 \sum_{1\leq i<j\leq n} h\pr{X_i,X_j}= \sum_{i=1}^{n-1} \sum_{j=i+1}^{n} a_{j-i}
U^jf =
  \sum_{j=2}^{n} \sum_{i=1}^{j-1} a_{j-i} U^jf = \sum_{j=2}^{n}
 \pr{j-1}^{3/2} U^jf
\end{equation}
where $f = \ind{[1/2, 1)}$. In order to have a better understanding of $U^jf$, we introduce the intervals
\begin{equation}
 I_{j,\ell}=\left[\frac{\ell-1}{2^j},\frac{\ell}{2^j}\right), j\geq 1, 1\leq \ell\leq 2^j.
\end{equation}

 \begin{Lemma}\label{lem:martingale_diff}
 The sequence $\pr{U^j\pr{f-1/2}}_{j\geq 1}$ is a   martingale difference  sequence with respect to the filtration
 $\pr{\Fca_{j}}_{j\geq 0}$, where $\Fca_j=\sigma\pr{I_{j,\ell},1\leq \ell\leq 2^j}$ and $\Fca_0=\ens{\emptyset,\Omega}$.
 \end{Lemma}
 \begin{proof}
We show by induction on $j\geq 1$ that
\begin{equation}\label{eq:expression_f_circ_Tn}
 f\circ T^j\pr{x}=\ind{\bigcup_{\ell=1}^{2^{j-1}}I_{j,2\ell} }\pr{x}.
\end{equation}
For $j=1$, notice that if $x\in [0,1/2)$, then
$f\circ T(x)=f(2x)=\ind{[1/2,1]}(2x)=\ind{[1/4,1/2)}\pr{x}$
and if $x\in [1/2,1)$, then $f\circ
T(x)=f(2x-1)=\ind{[1/2,1]}(2x-1)=\ind{[3/2,2)}\pr{2x}=
\ind{[3/4,1)}\pr{x}$ hence for each $x\in [0,1)$,
$f\circ T(x)=\ind{[1/4,1/2)}\pr{x}+\ind{[3/4,1)}\pr{x}$.
\\
Assume now that \eqref{eq:expression_f_circ_Tn} holds true
for some $j\geq 1$ and let us show that
\begin{equation}\label{eq:expression_f_circ_Tn+1}
 f\circ T^{j+1}\pr{x}=\ind{\bigcup_{\ell=1}^{2^{j}}I_{j+1,2\ell} }\pr{x}.
\end{equation}
By \eqref{eq:expression_f_circ_Tn} with $x$ replaced by $Tx$, we derive that
\begin{equation}
 f\circ T^{j+1}\pr{x}=\ind{\bigcup_{\ell=1}^{2^{j-1}}I_{j
 ,2\ell} }\pr{Tx}.
\end{equation}
If $x\in [0,1/2)$, then
$$
\ind{\bigcup_{\ell=1}^{2^{j-1}}I_{j
 ,2\ell} }\pr{Tx}=\ind{\bigcup_{\ell=1}^{2^{j-1}}I_{j
 ,2\ell} }\pr{2x}=\ind{\bigcup_{\ell=1}^{2^{j-1}}I_{j+1
 ,2\ell} }\pr{x}
$$
and \eqref{eq:expression_f_circ_Tn+1} holds,
and if $x\in [1/2,1)$, then
$$
\ind{\bigcup_{j=1}^{2^{j-1}}I_{j
 ,2\ell} }\pr{Tx}=\ind{\bigcup_{\ell=1}^{2^{j-1}}I_{j
 ,2\ell} }\pr{2x-1}=\sum_{\ell=1}^{2^{j-1}}\ind{I_{j,2\ell}}\pr{2x-1}
 =\sum_{\ell=1}^{2^{j-1}}\ind{I_{j,2\ell+2^j}}\pr{x}\\
 =\ind{\bigcup_{\ell=2^{j-1}+1}^{2^j}I_{j+1,2\ell}}\pr{x}
 $$\\
 hence \eqref{eq:expression_f_circ_Tn+1} also holds.\\
 By \eqref{eq:expression_f_circ_Tn}, it is clear that
 $U^jf$ is $\Fca_j$-measurable. Moreover,
 \begin{equation}
 \E{U^{j+1}\pr{f-1/2}\mid \Fca_{j}}=\sum_{\ell=1}^{2^{j}}
 \E{\ind{I_{j+1,2\ell}}-1/2\mid\Fca_{j} }=0.
 \end{equation}
 \end{proof}
Notice that $\binom{n}{2}^{-1}\sum_{1\leq i<j\leq n}\pr{h\pr{X_i,X_j} -\E{h\pr{X_i,X_j}}}=\sum_{j=1}^nd_{n,j}$, where
 \begin{equation}
 d_{n,j}=\binom{n}{2}^{-1}\pr{j-1}^{3/2} \pr{U^jf-\frac12},j\geq 2, d_{n,1}=0.
 \end{equation}
 Then $\pr{d_{n,j}}_{j\geq 1}$ is a martingale difference sequence with respect to the filtration $\pr{\Fca_j}_{j\geq 0}$ given
 as in Lemma~\ref{lem:martingale_diff}.  Recall that by \cite{MR358933}, if $\pr{d_{n,j}}_{n\geq 1,1\leq j\leq n}$ is an array of martingale differences,
 such that
 \begin{equation}\label{eq:cond1_mcleish}
 \max_{1\leq j\leq n}\abs{d_{n,j}}\to 0\mbox{ in probability},
 \end{equation}
 \begin{equation}\label{eq:cond2_mcleish}
 \mbox{there exists }M>0\mbox{ such that }\sup_{n\geq 1} \max_{1\leq j\leq n}\E{d_{n,j}^2}\leq M\mbox{ and }
 \end{equation}
\begin{equation}\label{eq:cond3_mcleish}
\sum_{j=1}^{n}d_{n,j}^2\to \sigma^2\mbox{ in probability},
\end{equation}
then $\sum_{j=1}^{n}d_{n,j}$ converges in distribution to a centered normal distribution with variance $\sigma^2$.\\
Noticing that $\abs{U^jf\pr{x}-1/2}=1/2$, we can see that
\eqref{eq:cond1_mcleish} and \eqref{eq:cond2_mcleish} are satisfied as well as
\eqref{eq:cond3_mcleish} with $\sigma^2=1/4$.\\
Letting $Y_n=\binom{n}2^{-1}\sum_{1\leq i<j\leq n}\pr{h\pr{X_i,X_j}-\E{h\pr{X_i,X_j}}}$, we thus get that $Y_n\to N\pr{0,1/4}$.
Expressing $Y_{2n}-Y_n$ as a sum of a martingale difference array, the same
argument as above gives that $Y_{2n}-Y_n$ converges in distribution to a
non-degenerated normal random variable hence $\pr{Y_n}_{n\geq 1}$
cannot converge in probability.
\end{proof}

\textbf{Acknowledgements.} The authors would like to thank the editor for
a suggested improvement of the counter-examples and for having help us to
clarify where the symmetry is needed.

\providecommand{\bysame}{\leavevmode\hbox to3em{\hrulefill}\thinspace}
\providecommand{\MR}{\relax\ifhmode\unskip\space\fi MR }
\providecommand{\MRhref}[2]{%
  \href{http://www.ams.org/mathscinet-getitem?mr=#1}{#2}
}
\providecommand{\href}[2]{#2}

\end{document}